\documentclass{proc-l}

\usepackage{amsmath}
\usepackage{amssymb,amsfonts,mathrsfs, amsthm, stmaryrd, tabu, graphicx, enumitem}
\usepackage{amssymb,amsfonts,mathrsfs}
\usepackage[colorlinks=true,linkcolor=blue,citecolor=blue]{hyperref}

\usepackage{amscd}
\usepackage{verbatim}
\usepackage{euscript}

\renewcommand{\subset}{\subseteq}

\newtheorem{theorem}{Theorem}[section]
\newtheorem{prop}[theorem]{Proposition}

\newtheorem{lemma}[theorem]{Lemma}

\newtheorem{maintheorem}{Theorem}

\numberwithin{equation}{section}

\newcommand{\gnorm}[1]{{\left\|\kern-0.24ex\left|{ #1 } \right|\kern-0.24ex\right\|}}

\DeclareMathOperator{\tr}{ tr }

\newcommand{\tiltau}{\widetilde{\tau}}
\newcommand{\gbar}{\overline{g}}
\newcommand{\hhat}{\widehat{h}}
\newcommand{\qhat}{\widehat{q}}
\newcommand{\fhat}{\widehat{f}}

\renewcommand{\hbar}{\overline{h}}

\newcommand{\bR}{\mathbb{R}}
\newcommand{\bB}{\mathbb{B}}
\newcommand{\bH}{\mathbb{H}}
\newcommand{\bN}{\mathbb{N}}

\newcommand{\sB}{\mathscr{B}}
\newcommand{\cQ}{\mathscr{Q}}
\newcommand{\fancy}{\mathscr{C}}
\newcommand{\fancym}{\mathscr{M}_{\mathrm{weak}}}

\newcommand{\wtil}{\widetilde{w}}

\newcommand{\grd}{g_0}
\newcommand{\bS}{\mathbb{S}}

\DeclareMathOperator{\Rc}{Rc}

\newcommand{\ghat}{\widehat{g}}

\newcommand{\Mbar}{\overline{M}}

\renewcommand{\)}{\textup{)}}

\newcommand{\defname}[1]{{\bfseries \itshape \boldmath #1}\/}

\begin{document}

\title{Low regularity Poincar\'e-Einstein metrics}


\author{Eric Bahuaud}
\address{Department of Mathematics,
Seattle University,
901 12th Ave,
Seattle, WA, 98122, United States}
\email{bahuaude(at)seattleu.edu}

\author{John M Lee}
\address{Department of Mathematics,
University of Washington,
Box 354350, 
Seattle, WA 98195-4350.}
\email{johnmlee(at)uw.edu}


\date{\today.}



\begin{abstract}
We prove the existence of a $C^{1,1}$ conformally compact Einstein metric on the ball that has asymptotic sectional curvature decay to $-1$ plus terms of order $e^{-2r}$ where $r$ is the distance from any fixed compact set.  This metric has no $C^2$ conformal compactification.
\end{abstract}

\maketitle

\section{Introduction}

A complete, noncompact Riemannian manifold $(M,g)$ is said to be 
\defname{conformally compact} if $M$ is the interior of a compact
manifold with boundary $\overline M$,
and there is a nonnegative function $\rho\colon \Mbar\to \bR$ such that
$\rho>0$ in $M$, $\rho=0$ to first order on $\partial M$, and
$\overline g =\rho^2 g$ has a continuous extension to a metric
on $\overline M$. It is said to have a \defname{$C^{k,\alpha}$ conformal compactification}
if the extended metric is of class $C^{k,\alpha}$ on $\overline M$, and is
said to be \defname{asymptotically hyperbolic} if
its sectional curvatures approach $-1$ at infinity.

Ever since the early 1980s,
there has been considerable interest in asymptotically hyperbolic Einstein 
metrics (now usually called \defname{Poincar\'e--Einstein metrics})
for both mathematical and physical reasons.
Mathematically, they are connected with global
conformal invariants of compact Riemannian manifolds,
and physically, they appear in the AdS/CFT correspondence of string theory and
as initial hypersurfaces for Einstein's equations in general relativity,
especially in the study of gravitational radiation.

The usual definition of a conformally compact metric is an extrinsic one:
One assumes the existence of a compact manifold with boundary $\overline M$,
an embedding $M\hookrightarrow \overline M$ whose image is the interior of $\overline M$,
and a conformal factor $\rho$ such that $\rho^2 g$ has 
an extension to $\overline M$ with suitable regularity. But from a geometric
point of view, it is interesting to explore the question of how far these conditions 
are determined by the intrinsic geometry of $(M,g)$.  In particular, what 
conditions on the behavior of $(M,g)$ at infinity are sufficient to guarantee
that it has a conformal compactification? What do these conditions tell us about
the regularity of the compactification?

There are some easy necessary conditions.
Suppose $(M,g)$ is a complete, noncompact $(n+1)$-dimensional Riemannian manifold.
If $M$ is to admit any conformal compactification, it must first of all contain
an \defname{essential subset}: This is a compact $(n+1)$-dimensional submanifold $K$ with 
smooth boundary, such that the outward normal exponential map from $\partial K$ 
is a diffeomorphism onto $\overline{M\smallsetminus K}$.
Under this hypothesis, $M$ can be embedded in a smooth compact manifold with boundary
$\overline M$ (diffeomorphic to $K$).

For Poincar\'e--Einstein manifolds, there is another necessary condition 
based on curvature decay. Suppose $(M,g)$ is a Poincar\'e--Einstein manifold 
with scalar curvature
equal to $-n(n+1)$.  (This is the scalar curvature of the 
$(n+1)$-dimensional hyperbolic metric with sectional curvature $-1$.)
If $g$ has a $C^2$ conformal compactification, then its  
sectional curvatures approach $-1$ to order $e^{-2r}$, where $r$ is the distance from 
any fixed compact subset of $M$.  We refer to this curvature property as 
\defname{quadratic hyperbolic curvature decay} (QHCD).
A natural question is whether  every Poincar\'e--Einstein metric with QHCD
has  a $C^2$ conformal compactification.

There have been several positive results in this direction. 
The first author \cite{B} showed that 
if $(M,g)$ is a noncompact Riemannian manifold with an essential subset $K$
and sectional curvatures approaching $-1$ to order $e^{-\omega r}$ with $\omega > 1$, together 
with similar decay on the covariant derivative of the curvature,
then $g$ admits a $C^{0,1}$ conformal compactification.  For related results, see \cite{BG, Gicquaud, HQS, ST}.

In \cite{BG}, the first author and Romain Gicquaud addressed the special case of Einstein metrics,
and showed that every Poincar\'e--Einstein manifold with an essential subset and QHCD
has a $C^{1,\alpha}$ compactification for any $\alpha\in (0,1)$. On the other hand, in a subsequent paper \cite{Gicquaud},
Gicquaud remarked that 
it does not seem unreasonable 
to believe that there exist 
Poincar\'e--Einstein metrics with QHCD that
have no $C^2$ conformal compactification.

The purpose of this note is to prove that this belief is justified.
We restrict attention to the case $n\ge 3$, because conformally Einstein metrics 
in dimensions $2$ and $3$ are hyperbolic and always have $C^\infty$ compactifications.

\begin{maintheorem} \label{thm:main} 
For each $n\ge 3$, there exists a  conformally compact Einstein
metric with QHCD  on the  $(n+1)$-dimensional ball 
that has a $C^{1,1}$ conformal compactification but no $C^2$ 
conformal compactification.
\end{maintheorem}

The proof of this theorem adapts the perturbative existence theorem of conformally compact Einstein metrics with prescribed conformal infinity on the ball by Graham and the second author \cite{GrahamLee}.  We begin by producing a one-parameter family 
of $C^{1,1}$ metrics on the boundary sphere that approach the standard round metric in $C^{1,1}$ norm.  The details of the construction ensure  that the regularity of the metric cannot be improved to $C^2$ by any coordinate or conformal change.  Next, by using the regularization technique and intermediate spaces introduced by Allen, Isenberg, Stavrov Allen and the second author in \cite{WAHM}, we  produce approximate solutions to the linearized Einstein equation with $C^{1,1}$ regularity.  By applying the inverse function theorem, we correct these approximate solutions to obtain actual Einstein metrics with the 
same conformal infinities. Finally, we show that the resulting Einstein metrics are in the class of 
``weakly asymptotically hyperbolic" metrics introduced in 
\cite{WAHM}, which implies that they have $C^{1,1}$ conformal compactifications.

This paper is structured as follows.  In the next section we describe the function spaces we will need in the subsequent analysis.  We extend the definitions and regularization procedure given in \cite{WAHM} to Lipschitz spaces.  In Section \ref{sec:extend} we describe our extension map from $C^{1,1}$ boundary metrics to weakly asymptotically hyperbolic metrics in the interior.  Section \ref{sec:C11-family} then takes up the construction of a one-parameter family of $C^{1,1}$ metrics that approach the standard round metric and possess no higher regularity.  In Section \ref{sec:einstein} we lay the foundation for the perturbation argument, and finally in Section \ref{sec:proof} we complete the proof of Theorem \ref{thm:main} by the inverse function theorem.

This work was supported by a grant from the Simons Foundation (\#426628, Eric Bahuaud).  The second author is happy to acknowledge the support of the 
Mathematical Sciences Research Institute and the Massachusetts Institute of
Technology during the time this work was being done.

\section{Analytic preliminaries}

\subsection{Function Spaces}
In this section we review the function spaces from \cite{Lee} and \cite{WAHM} that we will need in the subsequent analysis.  The main result of this section is the extension of the regularization procedure from \cite{WAHM} to Lipschitz spaces.

Consider the open unit ball $M = \bB^{n+1}$.  The function $\rho = \frac{1-|x|^2}{2}$ is a defining function for the sphere.  Choose the Euclidean metric as the background reference metric, hereafter denoted by $\hbar$, and denote the Poincar\'e metric by $h = \rho^{-2} \hbar$.

Following Chapter 2 of \cite{Lee}, we assume that we have covered $\overline{M}$ with a finite system of background coordinates, abbreviated by $(\theta, \rho) = (\theta^1,\dots,\theta^n,\rho)$ as in \cite{Lee}.  
Unless otherwise specified, we let Greek indices run from $1$ to $n$ and Latin ones from $1$ to $n+1$, with the understanding
that $\theta^{n+1} = \rho$. 

Let $B_r$ denote the ball of hyperbolic radius $r$ about the point $(0,1)$ in the upper-half space model of hyperbolic space.  Given any point $p_0 \in M$
with coordinates $(\theta_0, \rho_0) $ in some background chart, define 
a M\"obius parametrization $\Phi_{p_0}\colon  B_2 \to M$ by
\begin{equation*}
\Phi_{p_0}(x,y) = (\theta_0 + \rho_0 x, \rho_0 y).
\end{equation*}

Now choose (cf.~ \cite[Lemma 2.2]{Lee}) a countable collection of points $p_i$ so that $\{ \Phi_{p_i}(B_1)\}$ covers $M$ and the collection $\{ \Phi_{p_i}(B_2)\}$ is uniformly locally finite.  In short, we have covered the manifold by balls of a fixed intrinsic size and a M\"obius parametrization is an affine rescaling to a fixed ball where we perform computations.

Let $E$ denote a subbundle of the tensor bundle $T^{(r_1,r_2)}TM := (T M)^{\otimes r_1} \otimes (T^*M)^{\otimes r_2}$.  
Define the \defname{weight} of $E$ as $r:= r_2-r_1$.  In our application, $E$ is most frequently taken to be the bundle of symmetric covariant $2$-tensors, $\Sigma^2(M)$, for which $r=2$.

We will use two scales of H\"older spaces of sections of $E$ from \cite{WAHM}, which
we call plain and fancy.  
For $k \geq 0$ and $\alpha \in [0,1]$, the \defname{plain H\"older spaces} $C^{k,\alpha}(M; E)$ are defined by the norm
\[ \| u \|_{C^{k,\alpha}(M)} := \sup_{i} \| \Phi_i^* u \|_{C^{k,\alpha}(B_2)}, \]
where the H\"older norm on $B_2$ is taken with respect to the Euclidean metric in coordinates.  We also define a weighted norm $\|u\|_{C^{k,\alpha}_{\delta}(M)} = \| \rho^{-\delta} u \|_{C^{k,\alpha}(M)}$.  
It is on the $C^{k,\alpha}_\delta(M)$ spaces for $\alpha\in(0,1)$ 
that we have the Fredholm theorems from \cite{Lee}.  We frequently omit the target bundle in the notation.

The \defname{fancy H\"older spaces} $\fancy^{k,\alpha;m}(M;  E)$ of sections of  a tensor bundle $E$ of weight $r$ are defined 
as in \cite{WAHM}
by the norm 
\[ \gnorm{u}_{k,\alpha;m} := \sum_{l = 0}^m \| \overline{\nabla}^l u \|_{C^{k-l,\alpha}_{r+l}(M)}, \]
for $0 \leq m \leq k$,  where $\overline{\nabla}$ denotes the covariant derivative with respect to the background metric $\hbar$.  

Finally, we will also require H\"older/Lipschitz spaces of tensors on the compactification $\Mbar$.  Let $E$ denote a tensor bundle over $(\Mbar,\hbar)$.  For $k \in \bN_0$, $\alpha \in [0,1]$, let $C^{k,\alpha}(\Mbar; E)$ denote the usual H\"older space of tensor fields on $(\Mbar,\hbar)$.  In particular, observe that $C^{m-1,1}(\Mbar;E)$ is the space of sections with uniformly Lipschitz continuous derivatives to order $m-1$.

It is easy to check that equivalent norms result, and hence these spaces are unchanged,
if we replace the background metric $\overline h$ by any other
smooth metric $\overline g$ on the closed ball, and the hyperbolic metric by $g = \rho^{-2}\overline g$. We will occasionally
use this freedom to simplify some of the arguments.

We document a few facts about these spaces for the convenience of the reader.  

\begin{lemma}[Lemma 2.1 of \cite{WAHM}] \label{lemma:tensor-mapping-plain-Holder}\  
\begin{enumerate}
\item
If $E$ is a tensor bundle of weight $r$ over $\big(\overline M,\overline h\big)$,
then the following inclusion is continuous for any 
$k \in \bN_0$ and $\alpha \in (0,1)$:
\begin{equation*}
C^{k,\alpha}\big(\overline M;E\big) \hookrightarrow C^{k,\alpha}_r(M;E).
\end{equation*}
In particular, this means that any smooth vector field on $\overline M$
restricts to an element of $C^{k,\alpha}_{-1}(M;TM)$, and any smooth $1$-form on $\overline M$ 
restricts to an element of  
$C^{k,\alpha}_1(M;T^*M)$.
\item
For any 
$\mu, \mu' \in \bR$ and any tensor bundles $E_1$, $E_2$,
pointwise tensor product induces a continuous map
\begin{align} \label{eqn:tensor-mapping-plain-Holder}
C^{k,\alpha}_{\mu}(M; E_1) \times C^{k,\alpha}_{\mu'}(M; E_2) \longrightarrow C^{k,\alpha}_{\mu+\mu'}(M; E_1 \otimes E_2).
\end{align}
\end{enumerate}
\end{lemma}

The next lemma allows us to detect when a tensor in a plain H\"older space vanishes at the boundary.  See Lemma 3.7 of \cite{Lee} and Lemma 2.1 of \cite{WAHM}.

\begin{lemma} Let $E$ be a geometric tensor bundle of weight $r$ 
over $(\Mbar, \hbar)$. For $k \in \bN_0$, $\alpha \in (0,1)$, if $s = k+\alpha+r$, 
there is a continuous inclusion
\[ C^{k,\alpha}_{s}(M; E) \hookrightarrow C^{k,\alpha}(\Mbar; E). \]
As a consequence, if $s > k+\alpha+r$, 
then every section in $C^{k,\alpha}_{s}(M; E)$ has a continuous extension
to $\overline M$ that vanishes on the boundary.
\end{lemma}

We now document some important properties of the fancy H\"older spaces.

\begin{lemma}[Parts of Lemma 2.3 of \cite{WAHM}] \label{lemma:properties}
Suppose $\alpha \in [0,1)$ and $0 \leq m \leq k$.
\begin{enumerate}
\item\label{part:regular-in-fancy}
For $0\le m\le k$ and $\alpha\in[0,1)$,
we have $C^{k,\alpha}_{2+m}(M,\Sigma^2(M)) \subset \fancy^{k,\alpha;m}(M,\Sigma^2(M))$.
\item \label{lemma:embedding-Lipschitz} For $1 \leq m \leq k$, the following inclusion
is continuous:
\[ \fancy^{k,\alpha;m}(M) \hookrightarrow C^{m-1,1}(\Mbar). \]

\item \label{lemma:properties-ctyofder} 
The following maps are continuous:
\begin{align*}
\overline{\nabla}: &\fancy^{k,\alpha;m}(M) \to \fancy^{k-1,\alpha;m-1}(M),\\
M_\rho: &\fancy^{k,\alpha;m}(M)\to \fancy^{k,\alpha;m+1}(M),
\end{align*}
where $M_\rho$ represents 
multiplication by $\rho$.
\end{enumerate}
\end{lemma}

The next lemma shows how to detect whether a tensor is in the fancy H\"older spaces by 
looking purely at components in background coordinates.

\begin{lemma} \label{lemma:fancy-inclusion-criterion}
Suppose that $\tau = \tau_J^I d\theta^J \otimes \partial_{\theta^I}$ is a tensor field supported in the domain of a background coordinate chart $(\theta, \rho)$. Then 
\begin{equation*}
\tau \in \fancy^{k,\alpha;m}(M) \Longleftrightarrow \tau_J^I \in \fancy^{k,\alpha;m}(M) \text{ for all $I,J$}.
\end{equation*}
\end{lemma}

\begin{proof}
Since the fancy H\"older spaces are independent of the smooth
metric $\overline h$, we can without loss of generality assume that
$\overline h$ restricts to the Euclidean metric $(d\theta^1)^2 + \dots + (d\theta^n)^2 + d\rho^2$
in background coordinates 
on the support of $\tau$.  
In this case, the $\overline h$-covariant derivatives of $\tau$ are simply coordinate derivatives of its 
coefficients. 
The lemma now follows by simply comparing the norms of tensors and their component functions and noting  that the definition of the norm on $\fancy^{k,\alpha;m}$ includes the correct tensor weight.
\end{proof}

For the Lipschitz spaces $C^{k,1}_s(M)$, the next lemma gives an alternative
characterization in terms of background coordinates.

\begin{lemma}\label{lemma:plain-lipschitz}
Suppose $\tau = \tau_J^I d\theta^J \otimes \partial_{\theta^I}$ is a tensor field of weight $r$ supported in 
the domain of a background coordinate chart $(\theta, \rho)$.
Then $\tau\in C^{k,1}_s(M)$ if and only if 
each $\tau^I_J$ has $L^\infty$ 
partial derivatives up through order $k+1$, and
all of the following expressions are bounded:
\begin{equation}\label{eq:rho-s+jpartial}
\rho^{-s+r+j} \partial_{\theta^{i_1}}\dots \partial _{\theta^{i_j}}\tau^I_J,\quad 0\le j \le k+1,\quad  1\le i_s \le n+1. 
\end{equation}
If this is the case, the norm $\|\tau\|_{C^{k,1}_s(M)}$ is
uniformly equivalent to the supremum  of all the expressions in \eqref{eq:rho-s+jpartial}.
\end{lemma}

\begin{proof}
As in the previous lemma, we may assume
that 
$\overline h$ restricts to the Euclidean metric
in background coordinates 
on the support of $\tau$, so  $\overline h$-covariant derivatives of $\tau$ are coordinate 
derivatives of its component functions. 
By definition, 
$\tau \in C^{k,1}_s(M)$ if and only if $\Phi_i^*(\rho^{-s}\tau)\in C^{k,1}(B_2)$\
for 
any M\"obius parametrization $\Phi_i$.
Given a 
M\"obius parametrization 
$\Phi_i(x,y) = (\theta_i + \rho_i x,\rho_i y)$,
note that 
\begin{equation*}
\Phi_i^*(\rho^{-s}\tau ) = y^{-s} \rho_i^{-s+r} \tau^I_J(\theta_i + \rho_i x,\rho_i y)dx^J \otimes \partial_{x^I}.
\end{equation*}
Since $y^{-s}$ is a smooth function that is bounded above and below and has all derivatives
bounded on $B_2$, 
$\Phi_i^*(\rho^{-s}\tau)\in C^{k,1}(B_2)$ if and only if
\begin{equation*}
\rho_i^{-s+r} \tau^I_J(\theta_i + \rho_i x,\rho_i y)\in C^{k,1}(B_2).
\end{equation*}
The result now follows easily from the chain rule.
\end{proof}

We need the following generalization of Lemma 2.3(b) of \cite{WAHM}.

\begin{lemma}\label{lemma:Ivas-magic-Lipschitz}
Suppose $\tau$ is a tensor of weight $r$ in $C^{1,1}(\overline M)$
\(so that $\tau$ and $\overline\nabla\tau$ are Lipschitz continuous on $\overline M$\).
If $\tau=0$ on $\partial M$, then $\tau$ restricts to an element of $C^{1,1}_{r+1}(M)$,
with $\|\tau\|_{C^{1,1}_{r+1}(M)}\le C \|\tau\|_{C^{1,1}(\overline M)}$.
If in addition $\overline\nabla\tau=0$ on $\partial M$, 
then the restriction is in $C^{1,1}_{r+2}(M)$,
with $\|\tau\|_{C^{1,1}_{r+2}(M)}\le C' \|\tau\|_{C^{1,1}(\overline M)}$.
\end{lemma}

\begin{proof}
By means of a finite partition of unity, we reduce to the case where $\tau = \tau_J^I d\theta^J \otimes \partial_{\theta^I}$  is
supported in the domain of a single background coordinate chart with
coordinates $(\theta,\rho)$. 
The hypothesis $\tau\in C^{1,1}(\overline M)$ means 
each $\tau^I_J$ and its first and second partial derivatives 
in $(\theta,\rho)$ coordinates are uniformly bounded
by a multiple of $\|\tau\|_{C^{1,1}(\overline M)}$.

First suppose that $\tau=0$ on $\partial M$.
Then by the fundamental theorem of calculus, 
\begin{equation*}
|\tau^I_J(\theta,\rho)| =\left| \int_0^\rho \partial_\rho \tau^I_J(t,\theta)\,dt\right| 
\le \int_0^\rho C\|\tau\|_{C^{1,1}(\overline M)} = \rho C \|\tau\|_{C^{1,1}(\overline M)},
\end{equation*}
which shows that $\rho^{-1}\tau^I_J$ is bounded by $C \|\tau\|_{C^{1,1}(\overline M)}$. 
It then follows from Lemma \ref{lemma:plain-lipschitz} that
$\tau\in C^{1,1}_1(M)$ as claimed.

Now suppose in addition that $\overline\nabla \tau= 0$ on $\partial M$, which means
that all of the first $(\theta,\rho)$ derivatives of $\tau^I_J$ vanish on the boundary.
The Lipschitz condition on first derivatives then implies that all such
first derivatives are bounded by $C\rho\|u\|_{C^{1,1}(\overline M)}$.
Using the fundamental theorem of calculus as before, we conclude that $\rho ^{-2}\tau^I_J$ 
is bounded by a multiple of $\|\tau\|_{C^{1,1}(\overline M)}$, and then
Lemma \ref{lemma:plain-lipschitz} once again completes the proof.
\end{proof}

We will also need the regularization technique given by group-theoretic convolution introduced in \cite{WAHM}.  The half-space model of hyperbolic space $\bH = \bH^{n+1}$ is a group under $(\theta,\rho)\cdot (\theta',\rho') = (\theta+\rho \theta', \rho \rho')$.  For bounded integrable functions $\tau$ and $\psi$, at least one of which is compactly supported, define $\tau * \psi$ by
\[ (\tau*\psi)(q) := \int_{\bH} \tau(p) \psi( p^{-1} q) dV_{\bH}(p). \]

Here is a slight adaptation of a lemma from \cite{WAHM}.
\begin{lemma} \label{lemma:pre-reg}
Let $U$ and $V$ be open subsets of $\bH$.  Suppose that $\psi \in C^{\infty}_c(V)$ and that 
$\tau$ is a bounded integrable function supported in $U$.  Then 
\begin{enumerate}
\item $\mathrm{supp}( \tau * \psi) \subset UV = \{ pq: p \in U, q \in V \}$.
\item If $\tau$ is a real-valued function in $C^{m-1,1}(\overline{\bH})$ for some $m\ge 1$, then
\[ \tau * \psi \in \bigcap_{k \in \bN_0, \; \alpha \in (0,1)} \fancy^{k,\alpha;m}(M), \]
and
\[ \gnorm{\tau * \psi}_{k,\alpha;m} \leq C(k,\alpha,\mathrm{supp}\; \psi) \; \|{ \tau }\|_{C^{m-1,1}} \| \psi\|_{C^{k+1}}. \]
\item Suppose $\int_{\bH} \psi(q^{-1}) dV_{\bH} = 1$, and $\tau$ is a real-valued function in 
$C^{0,1}(\overline{\bH})$.  Then $\tau - \tau*\psi = O(\rho)$.
\end{enumerate}
\end{lemma}
\begin{proof}
The first claim is exactly as in Lemma 2.7(a) of \cite{WAHM}.  
The second claim is a minor modification of the proof 
of Lemma 2.7(b)
given in \cite{WAHM}: The proof there assumed that $\tau\in \fancy^{m,0;m}(\mathbb H)$, but
actually used only the fact that the background coordinate derivatives of $\tau$ up to order $m$ 
are \emph{uniformly bounded}, 
which is still true if $\tau$ is merely in $C^{m-1,1}(\overline M)$.  

Finally, the proof of the third claim in \cite{WAHM} uses precisely the Lipschitz regularity indicated in our hypothesis.
\end{proof}

We now establish a regularization result analogous of Theorem 2.6 of \cite{WAHM}.

\begin{prop} [Regularization] \label{prop:regularization-map}
Suppose $\tau$ is a tensor field in $C^{1,1}(\Mbar;\Sigma^2)$.  There exists a tensor $R(\tau)$ that lies in $\fancy^{k,\alpha;2}(M; \Sigma^2)$ for all $k \geq 0$ and $\alpha \in (0,1)$, depending linearly on $\tau$, such that
\[ R(\tau) - \tau \in C^{1,1}_{2+2}(M; \Sigma^2). \] 
Further, for each $k$ and $\alpha$ there exists a constant $C$ such that
\[ \gnorm {R(\tau) }_{k,\alpha;2} \leq C \|\tau\|_{C^{1,1}(\overline M)}. \]
\end{prop}
\begin{proof}
This requires only minor changes to the inductive proof of Theorem 2.6 in \cite{WAHM}.  In our case the induction requires two steps which we detail explicitly.

By finishing the argument with a partition of unity, it will be sufficient  to assume that $\tau$ is supported within a single background coordinate chart, which we write $(\theta, \rho)$ and use to identify with an open subset of the upper half-space $\bH$.  
By Lemmas \ref{lemma:fancy-inclusion-criterion} and \ref{lemma:plain-lipschitz} it suffices to work with component functions of $\tau$. 
To simplify notation, therefore, for the rest of the proof we assume $\tau$ is a real-valued function.
Let $\psi$ be a smooth function on $\bH$ that satisfies $\int_{\bH} \psi(p^{-1}) dV_{h}(p) = 1$ and that is compactly supported in a sufficiently small neighbourhood $V$ of $(0,1)$.

Initially $\tau \in C^{1,1}(\Mbar)$.  We begin by subtracting off a regularized version of its boundary value.  To this end, set $\tiltau = \tau * \psi$.  Now Lemma \ref{lemma:pre-reg} yields that $\tiltau \in \cap_{k,\alpha} \fancy^{k,\alpha;2}(M)$ and 
\[ \gnorm {\tiltau} _{k,\alpha;2} \leq C \|\tau\|_{C^{1,1}},  \]
and moreover, $\tau - \tiltau = O(\rho)$.  Applying Lemma \ref{lemma:properties}\eqref{lemma:embedding-Lipschitz} shows that $\tau - \tiltau \in C^{1,1}(\Mbar)$, and then Lemma \ref{lemma:Ivas-magic-Lipschitz} implies $\tau - \tiltau \in  C^{1,1}_1(M)$.

To proceed, set $u := \tau - \tiltau$; by the discussion above this lies in $C^{1,1}(\Mbar)\cap C^{1,1}_1(M)$.  Taking a $\rho$-derivative, we obtain
\[ w:= \frac{\partial u}{\partial \rho} \in C^{0,1}(\Mbar). \]
Setting $\wtil = w*\psi$, it follows that $\wtil \in \fancy^{k,\alpha;1}(M)$ for all $k$ by Lemma \ref{lemma:pre-reg}, and
thus $\wtil\in C^{2,\alpha}(M)\subset C^{1,1}(M)$ and
$\gnorm{\wtil}_{k,\alpha;1} \leq C \|w\|_{C^{0,1}}$.  Once more we obtain $w - \wtil = O(\rho)$.

Now $u - \rho \wtil \in C^{1,1}_1(M)$; we claim that $u - \rho \wtil \in C^{1,1}_2(M)$.  First, by Lemma \ref{lemma:properties}\eqref{lemma:properties-ctyofder}, $\rho \wtil \in \fancy^{k,\alpha;2}(M)$ and so $u - \rho \wtil \in C^{1,1}(\Mbar)$, and vanishes at $\rho=0$ because it is $O(\rho)$.  Thus any tangential derivative of the form 
\[ \frac{\partial}{\partial \theta^{\alpha}} (u - \rho \wtil  ),\]
will vanish at $\rho =0$ as well.  Any $\rho$-derivative may be written
\begin{align*}
\frac{\partial}{\partial \rho} (u - \rho \wtil  ) &= \partial_{\rho} u - \wtil - \rho \partial_{\rho} \wtil \\
&= w-\wtil + O(\rho) \\
&= O(\rho),
\end{align*}
which also vanishes at the boundary.  Thus by Lemma \ref{lemma:Ivas-magic-Lipschitz}, $u - \rho \wtil \in C^{1,1}_2(M)$ as claimed.  Set $\tiltau_2 = \tiltau + \rho \wtil$.
Also 
\[ \gnorm{\rho \wtil}_{k,\alpha;2} \leq C \gnorm{\wtil}_{k,\alpha;1} \leq C\|{ w}\|_{C^{0,1}} \leq C \| {u} \|_{C^{1,1}}. \]
So putting all of the estimates together,
\[ \gnorm {\tiltau_2}_{k,\alpha;2} \leq C \| {\tau}\|_{C^{1,1}}. \]
Now set $R(\tau) = \tiltau_2$.  This completes the proof.
\end{proof}

\section{An extension result}
\label{sec:extend}

Recall that $M$ is the open unit ball, $\hbar $ is the Euclidean metric on the closed ball $\overline M$,
and $h = \rho^{-2} \hbar$ is the hyperbolic metric on $M$.
The standard round metric is then $\hhat = \hbar|_{\bS^n}$.   
We describe a two-step extension procedure that takes metrics on $\partial M$ to asymptotically hyperbolic metrics on $M$.  
Let $\phi$ be a $C^{\infty}(\Mbar)$ bump function that is equal to $1$ on a neighbourhood
of $\partial M$ and supported in $A=M\smallsetminus \{0\}$. 
Let $P\colon A\to \partial M$ be the radial projection, and 
define
\[ E( \ghat ) := \phi P^* \ghat. \]
It is immediate that $E$ is a bounded linear map
\begin{equation}\label{eq:E-mapping}
E: C^{1,1}(\partial M, \Sigma^2 (\partial M)) \longrightarrow C^{1,1}(\Mbar, \Sigma^2(\Mbar)). 
\end{equation}

For the second step, we now extend to asymptotically hyperbolic metrics.  Define $T$ by
\[ T( \ghat) = h + \rho^{-2} R ( E ( \ghat - \hhat) ), \]
where $R$ is the regularization map from Proposition \ref{prop:regularization-map},
so that $R ( E ( \ghat - \hhat) )$ is obtained from $E ( \ghat - \hhat) $ locally by applying
the convolution operator twice.

Using the terminology and notation of \cite{WAHM}, 
we say a metric $g$ on $M$ is 
\defname{weakly $C^{k,\alpha}$ asymptotically hyperbolic} 
if $\gbar = \rho^2 g\in \fancy^{k,\alpha;m}(M,\Sigma^2(M))$ for some $k\ge 2$ and $m\ge 1$,
and $|d\rho|^2_{\gbar}=1$ on $\partial M$.  
The space of all such metrics for a given value of $m$ is denoted by
$\mathscr M^{k,\alpha;m}_{\text{weak}}$.

\begin{lemma}\label{lemma:properties-of-T}
Let $T$ be defined as above.
\begin{enumerate}
\item\label{part:Thhat}
$T(\hhat) = h$.  
\item\label{part:Tmapping}
For any $k \geq 2$ and $\alpha\in[0,1]$, $T$ is a continuous affine map of Banach spaces
\begin{equation} \label{eqn:defn-for-T}
T: C^{1,1}(\partial M, \Sigma^2 (\partial M)) \longrightarrow C^{k,\alpha}(M;\Sigma^2(M)).
\end{equation}
\item\label{part:TweaklyAH}
For any $\ghat\in C^{1,1}(\partial M, \Sigma^2 (\partial M))$
sufficiently close to $\hhat$,  $T(\ghat)$ 
is a metric on $M$ and lies in $\mathscr M^{k,\alpha;2}_{\text{weak}}$ for
all $k\ge 2$ and all $\alpha\in(0,1)$.
\end{enumerate}
\end{lemma}

\begin{proof}
Part \eqref{part:Thhat} is immediate from the definition.
Proposition \ref{prop:regularization-map} shows that $R$ is
a continuous linear map from $C^{1,1}(\Mbar, \Sigma^2 (\Mbar))$ to 
 $\fancy^{k,\alpha;2}(M; \Sigma^2(M))\subset C^{k,\alpha}_2(M;\Sigma^2(M))$, and then 
\eqref{part:Tmapping} follows easily from this and
\eqref{eq:E-mapping}.
To prove \eqref{part:TweaklyAH}, put $g = T(\ghat)$, which is in $C^{k,\alpha}(M)$ for all $k,\alpha$
by \eqref{part:Tmapping}, and then
Lemma \ref{lemma:properties} shows that $\gbar = \rho^2 g$ lies in $\fancy^{k,\alpha;2}(M)$. 
Because $T(\ghat)\to h$ in $C^{k,\alpha}(M)$ as $\ghat\to \hhat$, $g$ will be positive definite
on $M$ provided $\ghat$ is sufficiently close to $\hhat$.
Finally, the fact that $|d\rho|^2_{\gbar}=1$ on $\partial M$ follows from 
$\gbar = d\rho^2 + E(\ghat) + C^{1,1}_{2+2}(M)$, which is a consequence of 
Proposition \ref{prop:regularization-map}.
\end{proof}

\section{A family of $C^{1,1}$ metrics}
\label{sec:C11-family}

In this section we construct a one-parameter family of $C^{1,1}$ metrics on the unit sphere $\bS^n$ that approaches the round metric in $C^{1,1}$ norm.  To begin, consider $\bR^2 \times \bR^{n-2}$ with standard coordinates $(x^1, \cdots, x^n)$.

\begin{lemma} \label{lemma:R2D2}
Let $n \geq 3$.  There exists a $C^{1,1}$ Riemannian metric $k$ on $\bR^n$ of the form
\[ k = \delta + e, \]
where $\delta$ is the Euclidean metric and $|e|_{\delta}^2 = O(|x|^2)$.
No conformal multiple of this metric can be improved to class $C^2$
by any change of coordinates.
\end{lemma}

\begin{proof}
Let $f\colon[0,\infty)\to\bR$ be a smooth function that satisfies
\begin{equation*}
f(0)= 1, \quad f'(0) = 0, \quad f''(0) \ne 0, \quad f'''(0) \ne 0,
\end{equation*}
(for example, $f(t) = 1+t^2+t^3$),
and define a $C^{1,1}$ function $\fhat\colon\bR\to\bR$ by
\begin{equation*}
\fhat(t) = \begin{cases}
f(t), &t\ge 0,\\
1, &t<0.
\end{cases}
\end{equation*}
Then let $k_2$ be the following warped-product metric on $\bR^2$:
\begin{equation*}
k_2 = (dx^1)^2 + \fhat\big(x^1\big)^2\, (dx^2)^2,
\end{equation*}
and let $k$ be the product metric on $\bR^n = \bR^2\times\bR^n$ defined by
\begin{equation*}
k = k_2 \oplus \delta_{n-2},
\end{equation*}
where $\delta_{n-2} = \sum_{\alpha=3}^n (dx^\alpha)^2$ is the Euclidean metric on $\bR^{n-2}$.

We will write the Riemann, Ricci, and scalar curvatures of $k$ as $Rm$, $Rc$, and $S$.  Recall that the Schouten tensor $P$, Cotton tensor $C$, 
and Weyl tensor $W$ are defined by the formulas
\begin{align*}
P_{ij} &= \frac{1}{n-2} \left( Rc_{ij} - \frac{S}{2(n-1)} g_{ij} \right) \\
C_{ijk} &= P_{ij,k} - P_{ik,j} \\
W_{ijkl} &= R_{ijkl} - \left( P \owedge g \right)_{ijkl}\,
\end{align*}
where $\owedge$ is the Kulkarni-Nomizu product,
defined for symmetric $2$-tensors $a,b$ by 
\begin{equation*}
(a\owedge b)_{ijkl} = a_{il}b_{jk} + a_{jk}b_{il} - a_{ik}b_{jl} - a_{jl}b_{ik}.
\end{equation*}

In the computations that follow, the indices $1$ and $2$ refer to $x^1$ and $x^2$, and Greek indices refer to the 
coordinates $x^3,\dots,x^n$.
When $x^1>0$, the nonzero Christoffel symbols of $k_2$ are
\begin{equation}\label{eq:Gammxyy}
\Gamma_{12}^2 = \Gamma_{21}^2 = \frac{f'}{f}, \qquad 
\Gamma_{22}^1 = -ff',
\end{equation}
and its nonzero curvature components are 
\begin{equation}\label{eq:Rxyyx}
R_{1221} = R_{2112} = - R_{1212} = - R_{2121} = -ff'' .
\end{equation}
For the product metric $k$, 
the Christoffel symbols $\Gamma_{ij}^k$ and
curvature components $R_{ijkl}$ are all zero
if any of the indices is greater than $2$ or if $x^1<0$, 
and the nonzero ones  when $x^1>0$ are 
given by \eqref{eq:Gammxyy} and \eqref{eq:Rxyyx}.
Thus when $x^1>0$, the metric $k$ has curvatures given by
\begin{align*}
Rc &= - \frac{f''}{f} (dx^1)^2 - f f'' (dx^2)^2 \\
S &= - 2 \frac{f''}{f},\\
P &= -\frac{1}{n-1} \frac{f''}{f} (dx^1)^2 
- \frac{1}{n-1} f f''(dx^2)^2.
\end{align*}

To analyze the effect of a conformal change, 
consider the following component of the Weyl tensor for $x^1>0$:
\begin{align*}
W_{1221} &= R_{1221} - P_{11}g_{22} - P_{22}g_{11} + 2 P_{12}g_{12}\\
&= -ff'' + \frac{1}{n-1}\,\frac{f''}{f} f^2 + \frac{1}{n-1}ff''\\
&= -\frac{n-3}{n-1} f f''.
\end{align*}
When $n>3$, this is discontinuous at $x^1=0$, and will still be
discontinuous after multiplying $k$ by any conformal factor, so
no conformal multiple of $k$ can be improved to class $C^2$ in
any neighborhood of the origin by any
choice of coordinates.

For the $n=3$ case, we need to check the Cotton tensor.
For $x^1>0$, we have
\begin{align*}
P_{22,1} &= \partial_1 P_{22} - 2\Gamma_{12}^2 P_{22} \\
&= \left( - \frac{1}{n-1} f f'' \right)' - 2 \left(- \frac{1}{n-1} f f'' \right)\left( \frac{f'}{f} \right) \\
&= \frac{1}{n-1} \left( f' f'' - f f''' \right).
\end{align*}  
A similar computation shows that $P_{21,2}=\partial_2 P_{21}-\Gamma _{22}^1 P_{11} - \Gamma_{12}^2 P_{22}=0$, so 
$C_{221} = P_{22,1} - P_{21,2} = P_{22,1}$, which is discontinuous
at $x^1=0$. Since the Cotton tensor is invariant in $3$ dimensions
under a conformal change of metric, this shows that $k$
cannot be smoothed by a conformal or coordinate change in that dimension either.
\end{proof}

The pullback of the round metric on $\bS^n$ via the inverse of stereographic projection is given by
\[ \grd = \frac{4}{(1 + |x|^2)^2} \delta, \]
where $\delta =(dx^1)^2 + \cdots +(dx^n)^2 $ is the Euclidean metric on $\bR^n$.
Let $\phi: \bR^n \to [0,1]$ be a smooth radial cutoff function such that $\phi \equiv 1$ for $|x| \leq 1$ and $\phi$ is supported in $|x| \leq 2$.  Also, for $\lambda > 0$ introduce a dilation operator $D_{\lambda}: \bR^n \to \bR^n$ by $D_{\lambda}(x) = \lambda x$.  

\begin{prop}\label{prop:glambda}
For any $\lambda > 0$, the metric $g_{\lambda}$ on $\bR^n$ defined by
\begin{equation}
g_{\lambda} :=  \frac{4 }{(1+|x|^2)^2}\left(  \frac{1}{\lambda^2} \phi(x) D_{\lambda}^* k + (1-\phi(x)) \delta \right),
\end{equation}
pulls back via stereographic projection to a $C^{1,1}$ metric $\ghat_\lambda$ on $\bS^n$.  Moreover, $\ghat_{\lambda} \to \grd$ in the $C^{1,1}$ norm \(measured with respect to $g_0$\) as $\lambda \to 0$.
No conformal multiple of $\ghat_\lambda$ can be improved to class $C^2$
by any change of coordinates.
\end{prop}
\begin{proof}
Since 
\[ \frac{1}{\lambda^2} D_{\lambda}^* k = \delta + \big(\fhat( \lambda x^1 )^2-1\big) (dx^2)^2, \]
one finds that
\[ g_{\lambda} = g_0 + \frac{4}{(1+|x|^2)^2}\phi(x) \big(\fhat( \lambda x^1 )^2-1\big) (dx^2)^2. \]
By the construction of $\fhat$ given in Lemma \ref{lemma:R2D2}, $g_{\lambda}$ is $C^{1,1}$ for all $\lambda > 0$.  Since $\phi$ is compactly supported, $g_{\lambda}$ pulls back to a $C^{1,1}$ metric on $\bS^n$ under stereographic projection.  

Since $\fhat( \lambda x^1 )^2-1$ and its first two
coordinate derivatives are uniformly bounded by a multiple of $\lambda^2$ on the support of $\phi$,
it is straightforward to check that  $\ghat_{\lambda} \to g_0$ in $C^{1,1}$.
\end{proof}

\section{Einstein metrics}
\label{sec:einstein}

The Einstein equation is not elliptic, and thus following the strategy in \cite{GrahamLee, Lee} we will work with a gauge-broken  equation.  Let $g_0$ denote a conformally compact reference metric and let $\Delta_{gg_0}(Id)$ denote the harmonic map Laplacian from $(M,g)$ to $(M,g_0)$.  Let $\delta_g$ denote the divergence operator of $g$ and $\delta_g^*$ its formal adjoint.  Then the equation
\begin{equation} \label{eqn:defn-of-Q}
Q(g,g_0) = \Rc(g) + ng - \delta_{g_0}^*( \Delta_{g g_0} Id ) = 0
\end{equation}
is a quasilinear elliptic equation for $g$.  The differentiability of $Q$ on the spaces $C^{k,\alpha}_{\mu}(M; \Sigma^2(M))$ was established in \cite[Lemma 8.4]{Lee}.  Take $M = \bB^{n+1}$ and $g_0 = h$.  The linearization of equation \eqref{eqn:defn-of-Q} in $g$  at the hyperbolic metric is
\[ (D_1 Q)_{(h,h)} v = \frac{1}{2} (\Delta^h_L  + 2n ) v, \]
where $\Delta_L$ is the Lichnerowicz Laplacian.  Using the inverse function theorem we will show that equation \eqref{eqn:defn-of-Q} has an asymptotically hyperbolic solution with any prescribed $C^{1,1}$ conformal infinity sufficiently close to the standard round metric.  A maximum principle argument (see Proposition \ref{prop:max-principle-arg} below) then allows us to conclude that this metric is Einstein.

Here is the basic analytic fact we will need.
\begin{prop} \label{prop:Jack-Fredholm-result}
\[ \Delta_L^h + 2n: C^{2,\alpha}_{\mu}(M; \Sigma^2(M)) \longrightarrow C^{0,\alpha}_{\mu}(M;\Sigma^2(M)) \]
is an isomorphism if and only if $\mu \in (0,n)$.
\end{prop}
\begin{proof}
This is a simple application of Theorem C of \cite{Lee} applied to the operator $P := \Delta_L^h + 2n$.  To see this, first note by Proposition D of \cite{Lee}, the indicial radius of $P$ is $R = \frac{n}{2}$.  So Theorem C allows us to conclude that $P$ is Fredholm if $| \mu - \frac{n}{2} | < \frac{n}{2}$, or $\mu \in (0,n)$.  Moreover the Fredholm index is zero, and the kernel of $P$ is equal to the $L^2$ kernel of $P$.
However $P$ has trivial $L^2$ kernel, as shown in the proof of Theorem A of \cite{Lee}.
\end{proof}

To find a solution to the gauge-broken Einstein equation, we begin by 
showing that $Q( T(\ghat), T(\ghat) )$ already decays to second order.

\begin{lemma} \label{lemma:Q-first-approx} For any $\ghat \in C^{1,1}(\partial M; \Sigma^2(\partial M))$,
$Q( T(\ghat), T(\ghat) ) \in C^{0,\alpha}_2(M).$
\end{lemma}

\begin{proof}
For convenience set $g = T(\ghat)$.  Since both arguments of $Q$ are the same, the gauge term is zero.  We thus 
need to check that $\Rc(g) + ng$ lies in the prescribed space.  Since $g$ is in the image of $T$, 
Lemma \ref{lemma:properties-of-T} shows that $g\in\fancym^{2,\alpha;2}$.  
Therefore by Theorems 1.3 and 1.4 of \cite{WAHM}, it suffices to check that
\[ |d \rho|^2_{\gbar} - 1 - \frac{2}{n+1} \rho \Delta_{\gbar} \rho = O(\rho^2), \]
for then $\Rc(g) + ng \in C^{0,\alpha}_2(M)$.

Since $ g = h + \rho^{-2} RE( \ghat-\hhat )$,  near $\partial M$ (i.e., where the cutoff $\phi$ in the definition of $E$ satisfies $\phi \equiv 1$) we have
\begin{align*}
 \gbar &= \hbar +  R P^*( \ghat-\hhat ) \\
 &= (d\rho^2 + P^*\hhat) + (P^*\ghat - P^*\hhat) + Z \\
 &= d\rho^2 + P^*\ghat + Z,
\end{align*}
where $Z \in C^{1,1}_{4}(M)$ by Lemma \ref{prop:regularization-map}. 
Writing $Z$ in background coordinates as $Z = Z_{ij}d\theta^i\,d\theta^j$ (with Roman indices 
running from $1$ to $n+1$), we see that $Z_{ij} = O(\rho^2)$ and $\partial_\rho Z_{ij} = O(\rho)$.  
Using a $\rho$ index to denote the $\theta^{n+1}=\rho$ direction and Greek indices
to denote $\theta^1,\dots,\theta^n$, we can write the 
components of $\gbar$ as
\begin{align*}
\gbar_{\rho\rho} & = 1 + Z_{\rho\rho},\\
\gbar_{\rho\beta} &= Z_{\rho\beta},\\
\gbar_{\alpha\beta} &= \ghat_{\alpha\beta} + Z_{\alpha\beta},
\end{align*}
with $\partial_\rho \ghat_{\alpha\beta}\equiv 0$.
It follows that $\gbar^{\rho\rho} = 1+O(\rho^2)$
and the Christoffel symbols of $\gbar$ are all $O(\rho)$.
Therefore,
\begin{equation*}
|d \rho|^2_{\gbar} - 1 - \frac{2}{n+1} \rho \Delta_{\gbar} \rho 
= \gbar^{\rho\rho} - 1+ \frac{2}{n+1} \rho \gbar^{ij} \overline\Gamma^\rho_{ij} = O(\rho^2),
\end{equation*}
as claimed.
\end{proof}

We also need the following lemma.
\begin{lemma}
Let $\ghat \in C^{1,1}(\partial M; \Sigma^2(\partial M))$, and let $r \in C^{2,\alpha}_2(M;\Sigma^2(M))$. Then
\[ Q( T(\ghat) + r, T(\ghat) ) \in C^{0,\alpha}_2(M;\Sigma^2(M)). \]
\end{lemma}
\begin{proof}
For convenience set $g = T(\ghat)$.  Now
\[ Q( g + r, g) - Q(g,g) = \int_0^1 (D_1 Q)_{(g+sr,g)}( r) ds, \]
where $D_1Q$ is the derivative of $Q$ with respect to its first argument.
By Lemma \ref{lemma:Q-first-approx}, $Q(g,g) \in  C^{0,\alpha}_2(M;\Sigma^2(M))$, so it suffices to understand the term inside the integral.
The explicit formula for $(D_1 Q)_{(g+sr,g)} (r)$ appears as equation (2.15) in \cite{GrahamLee}.  Since $g+sr$ and $g$ lie in $C^{2,\alpha}(M)$, from this formula one checks that $(D_1 Q)_{(g+sr,g)}$ is a uniformly degenerate
operator with coefficients that (at worst) lie in $C^{0,\alpha}(M)$.  Combined with equation \eqref{eqn:tensor-mapping-plain-Holder} of Lemma \ref{lemma:tensor-mapping-plain-Holder} we conclude $(D_1 Q)_{(g+sr,g)}$ maps $C^{2,\alpha}_2(M)$ to $C^{0,\alpha}_2(M)$, completing the proof.
\end{proof}

\section{Proof of Theorem \protect\ref{thm:main} }
\label{sec:proof}

We now begin the proof of Theorem \ref{thm:main}.  To solve $Q(g,g_0) = 0$, we will apply the Banach inverse function theorem in the space $C^{2,\alpha}_2(M; \Sigma^2(M))$.  To this end, define an open subset
\begin{align*}
\sB & \subset C^{1,1}(\partial M; \Sigma^2(\partial M)) \times C^{2,\alpha}_{2}(M; \Sigma^2(M)) \; \; \mbox{by} \; \\
\sB & := \left\{ (\ghat, r): \ghat, T(\ghat), T(\ghat) + r \; \mbox{are positive definite} \right\}.
\end{align*}

Now define a map:
\begin{align*}
\cQ: C^{1,1}(\partial M; \Sigma^2(\partial M)) &\times C^{2,\alpha}_{2}(M; \Sigma^2(M))\\
 &\longrightarrow  C^{1,1}(\partial M; \Sigma^2(\partial M)) \times C^{0,\alpha}_{2}(M; \Sigma^2(M))
\end{align*}
by
\begin{align*}
\cQ (\ghat, r) &= \left( \ghat, Q\left( T(\ghat) + r, T(\ghat) \right) \right).
\end{align*}

Observe that $\cQ( \hhat, 0 ) = (\hhat, 0 )$.  The explicit calculation in the proof of Theorem A of \cite{Lee} shows that the linearization of $\cQ$ at $(\hhat, 0)$ is given by
\begin{align*}
D \cQ_{(\hhat,0)} (\qhat,r) &= ( \qhat, D_1 Q_{(h,h)}( DT_{\hhat} \qhat + r) + D_2 Q_{(h,h)} (DT_{\hhat} \qhat) ) \\
&= ( \qhat, (\Delta_L + 2n) r + K \qhat  ),
\end{align*}
where $K$ is defined by
\[ K \qhat := D_1 Q_{(h,h)}( DT_{\hhat} \qhat ) + D_2 Q_{(h,h)} (DT_{\hhat} \qhat) ). \]
By Proposition \ref{prop:Jack-Fredholm-result} if $\mu = 2$ and $n \geq 3$ we obtain that $\Delta_L + 2n$ is invertible.  So $D \cQ$ is invertible.  

The Banach inverse function theorem now shows that there is a neighbourhood of $(\hhat, 0)$ in $\sB$ on which $\cQ$ has a smooth inverse.  
We choose a boundary metric $\ghat=\ghat_\lambda$ given by 
Proposition \ref{prop:glambda} with $\lambda$ sufficiently small.
Thus there is a solution 
$r \in C^{2,\alpha}_2(M;\Sigma^2(M))$ such that
$\cQ( \ghat, r ) = (\ghat, 0)$.

Set $g = T(\ghat) + r$ and $g_0 = T(\ghat)$, so $Q( g, g_0 ) = 0$.  
Lemma \ref{lemma:properties-of-T} shows that $\rho^2 T(\ghat)\in \fancy^{2,\alpha;2}(M;\Sigma^2(M))$,
and $\rho^2 r \in C^{2,\alpha}_{2+2}(M;\Sigma^2(M))\subset \fancy^{2,\alpha;2}(M;\Sigma^2(M))$
by Lemma  \ref{lemma:properties}\eqref{part:regular-in-fancy}. 
Thus $\gbar\in \fancy^{2,\alpha;2}(M;\Sigma^2(M))\subset C^{1,1}(\Mbar;\Sigma^2(\Mbar))$, which means that 
$g$ has a $C^{1,1}$ conformal compactification.
Moreover, by \cite[Theorem 1.4]{WAHM}, $g$ has QHCD.   
By restricting to the boundary $T \partial M$ we find
\begin{align*}
 \gbar|_{T\partial M} &= \hbar|_{T\partial M} +  R E ( \ghat-\hhat )|_{T\partial M} \\
 &=  \hhat + E(\ghat - \hhat)|_{T\partial M} + Z|_{T\partial M} \\
 &= \ghat,
\end{align*}
where $Z \in C^{1,1}_4(M;\Sigma^2(M))$, and thus $Z|_{T\partial M} = 0$.  So $g$ has the prescribed conformal infinity $\ghat$.  
There can be no $C ^2$ conformal compactification of $g$, because it would induce a
smooth structure on $\partial M$ in which some conformal multiple of $\ghat$ is
of class $C^2$, which is ruled out by Proposition \ref{prop:glambda}.

The proof of Theorem \ref{thm:main} is then completed once we show that $g$ is Einstein.

\begin{prop} \label{prop:max-principle-arg}

For $\ghat$ sufficiently close to $\hhat$ in $C^{1,1}$ norm, the resulting solution $g$ of the linearized Einstein equation $Q(g,g_0) = 0$ is an Einstein metric.
\end{prop}
\begin{proof}
The proof follows Lemma 2.2 of \cite{GrahamLee} closely.  
To better match the notation of \cite{GrahamLee}, set $t = g_0$, and let 
$\omega$ be the gauge $1$-form
\begin{equation*}
\omega = (\Delta_{g t} Id)^\flat = gt^{-1}\delta_g\big( t - \tfrac 1 2 \tr(g^{-1} t) g\big).
\end{equation*}
The key idea is to show that $\omega$
vanishes by the maximum principle.  
By virtue of Lemma \ref{lemma:properties-of-T},
$t \in C^{3,\alpha}(M)$ for any $\alpha \in (0,1)$. This extra regularity is used in the maximum principle argument below.

The map $g\mapsto \Rc(g)$ is continuous from $C^{2,\alpha}(M;\Sigma^2(M))$ to $C^{0,\alpha}(M;\Sigma^2(M))$.  Since $T$ is continuous from $C^{1,1}(\partial M, \Sigma^2 (\partial M))$ to $C^{2,\alpha}(M;\Sigma^2(M))$, we can control the $L^{\infty}$ norm of $\Rc(g) - \Rc(h) = \Rc(g) + nh$ in terms of the $C^{1,1}$-norm of $\ghat-\hhat$.  Thus for $\ghat$ sufficiently close to $\hhat$, $\Rc(g)$ is strictly negative on $M$.

As in \cite[Lemma 2.2]{GrahamLee}, the Bianchi identity applied to $Q(g,t) = 0$ leads to the differential inequality
\[ \Delta^g |\omega|_g^2 \leq 2 K |\omega|^2_g, \]
for some negative constant $K$.
Since $t \in C^{3,\alpha}(M;\Sigma^2(M))$, $\delta^g t \in C^{2,\alpha}(M;\Sigma^1(M))$, and thus the function
$|\omega|^2_g \in C^{2,\alpha}(M)$ is bounded.  The generalized maximum principle (Theorem 3.5 of \cite{GrahamLee}) applies 
since $g$ is $C^{1,1}$ conformally compact and $|\omega|^2_g$ is bounded and $C^2$ in $M$, and we deduce $|\omega|^2_g = 0$ exactly as in the proof of Lemma 2.2.  But $\omega = 0$ implies
\[ \Rc(g) + ng = 0, \]
as required. 
\end{proof}

\end{document}